\documentclass[11pt,a4paper,twoside]{amsart}

\usepackage{amsmath,amsfonts,amsthm,amsopn,color,amssymb,enumitem,cite,soul}
\usepackage{palatino}
\usepackage{graphicx}
\usepackage[normalem]{ulem}
\usepackage[colorlinks=true,urlcolor=blue,citecolor=red,linkcolor=blue,linktocpage,pdfpagelabels,bookmarksnumbered,bookmarksopen]{hyperref}
\hypersetup{urlcolor=blue, citecolor=red, linkcolor=blue}
\usepackage[left=2.61cm,right=2.61cm,top=2.72cm,bottom=2.72cm]{geometry}

\usepackage[hyperpageref]{backref}

\usepackage[colorinlistoftodos]{todonotes}
\makeatletter
\providecommand\@dotsep{5}
\def\listtodoname{List of Todos}
\def\listoftodos{\@starttoc{tdo}\listtodoname}
\makeatother

\newcommand{\eps}{\varepsilon}

\newcommand{\R}{\mathbb{R}}

\newcommand{\RN}{{\mathbb{R}^N}}

\newcommand{\RT}{{\mathbb{R}^3}}

\newcommand{\B}{{\mathcal B}}

\DeclareMathOperator{\supp}{supp}
 
\renewcommand{\le}{\leslant}
\renewcommand{\ge}{\geslant}
\renewcommand{\a }{\alpha }

\newcommand{\g }{\gamma }

\newcommand{\n }{\nabla }

\newcommand{\Hr}{H^1_r(\RT)}
\newcommand{\HT}{H^1(\RT)}

\newcommand{\N}{\mathbb{N}}

\renewcommand{\o}{\omega}

\newcommand{\D }{{\mathcal D}^{1,2}(\RT)}
\newcommand{\Dr }{{\mathcal D}^{1,2}_r(\RT)}

\newcommand{\irt }{\int_{\RT}}

\def\bbm[#1]{\mbox{\boldmath $#1$}}
\newcommand{\beq }{\begin{equation}}
\newcommand{\eeq }{\end{equation}}

\renewcommand{\le}{\leqslant}
\renewcommand{\ge}{\geqslant}
\newcommand{\dis}{\displaystyle}

\newtheorem{theorem}{Theorem}[section]
\newtheorem{lemma}[theorem]{Lemma}

\newtheorem{proposition}[theorem]{Proposition}
\newtheorem{remark}[theorem]{Remark}
\newtheorem{corollary}[theorem]{Corollary}

\title[K-G-M system: limit case]{Finite energy standing waves \\for the Klein-Gordon-Maxwell system: \\the limit case}

\author[A. Azzollini]{Antonio Azzollini}
\address{Dipartimento di Matematica, Informatica ed Economia, Universit\`a degli
Studi della Basilicata,
\newline\indent
Via dell'Ateneo Lucano 10, I-85100
Potenza, Italy}
\email{antonio.azzollini@unibas.it}

\subjclass[2010]{35J20, 35Q60}
\keywords{Klein-Gordon-Maxwell system, standing waves}

\begin{document}
	
		\begin{abstract}
		In this paper we consider the Klein-Gordon-Maxwell system in the electrostatic case, assuming the fall-off large-distance requirement on the gauge potential. We are interested in proving the existence of finite energy (and finite charge) standing waves, having the phase corresponding to the mass coefficient in the Klein-Gordon Lagrangian. 
	\end{abstract}
\maketitle

\section{introduction}
As it is well known, standing waves solving the nonlinear Klein-Gordon-Maxwell system in the electrostatic case can be obtained from the system 
	\begin{equation}    \label{electrostatic}
	\left\{
	\begin{array}{ll}
	-\Delta u + [(m^2-\o^2)+e(2\o-e\phi)\phi] u-u^{p-1}=0 & \hbox{in }\RT,
	\\
	-\Delta \phi=e(\o-e\phi) u^2 & \hbox{in }\RT,\\
	u>0,\;\phi>0&\hbox{in }\RT
	\end{array}
	\right.
	\end{equation}
for $m, e, \o >0$ and $p>1$ (we refer to \cite{BF} for the derivation of the system). The interest in these equations rests on the gauge theory from which they come.\\
A couple $(u,\phi)$ solving \eqref{electrostatic} originates on one hand a matter field having the form of a standing wave
\begin{equation}\label{eq:matter}
\psi(x,t)= u(x)e^{-i\o t},\; u>0
\end{equation}
on the other the electromagnetic field 
	\begin{equation}\label{eq:elecmag}
		({\bf E}(x,t),{\bf H}(x,t))=(-\n \phi(x), {\bf 0})
	\end{equation}
interacting and influencing each other (see \cite{BF2}).\\ 
The physical relevance of this model is strengthened by the property of {\it localization} possessed by fields as in \eqref{eq:matter}.
Indeed, by the invariance of the original Lagrangian with respect to the Poincar\'e group of tranformations, we are allowed to put a standing wave in motion by means of a Lorentz boost, obtaining a solitary wave behaving like a relativistic particle.
In this sense, the system provides a relativistic consistent model for the description of the interaction between a particle embedded in the electromagnetic field generated by itself. \\
We point out that, as it is showed in \cite{BF2}, by the gauge invariance of the original Lagrangian with respect to transformations of the type
\begin{align*}
u(x)e^{-i\o t}&\mapsto u(x)e^{-i(\o t-\chi(t))}\\
\phi(x)&\mapsto \phi(x)-\frac{\partial}{\partial t}\chi(t)
\end{align*} 
where $\chi\in C^{\infty}(\R)$, starting from a solution $(u,\phi)$ of \eqref{electrostatic} and considering tranformations of the type $\chi(t)=ct$ for $c\in\R$, we can obtain standing waves $\psi(x,t)=u(x)e^{-i\tilde \o t}$ for an arbitray $\tilde\o\in\R$, remaining the electromagnetic field unvaried. \\
This gauge freedom is avoided by requiring, for example, a further condition on the behaviour of the gauge potential $\phi$ at infinity. In particular, in our project we are interested in potentials $\phi$ satisfying the so called large-distance fall-off requirement
	\begin{equation}\label{FO}\tag{{$\mathcal FO$}}
		\phi(x)\to 0\hbox{ as } |x|\to +\infty,
	\end{equation} 
in view to show that $\phi$ actually behaves like a Coulomb potential at infinity (see Remark \ref{re:cou}).

The first existence and multiplicity results concerning problem \eqref{electrostatic} were obtained in \cite{BF}. These results were later improved in \cite{DM} and \cite{APP}, while nonexistence theorems were proved in \cite{DM2}. The Klein-Gordon-Maxwell coupling was also considered in a bounded domain with various boundary conditions for instance in \cite{dps,dps1}. \\

In all the quoted papers, the relation between the mass coefficient $m$ and the wave phase $\omega$ plays an important role to establish the existence of solutions.\\
In particular, putting together the  results contained in \cite{APP,BF,DM}, we have the following
	\begin{theorem}\label{th:known}
		Let $p\in (2, 6)$ and assume that $0 < \omega < m g(p)$ where
	\begin{equation*}
		g(p) =
		\left\{
		\begin{array}{ll}
		\sqrt{(p-2)(4-p)} & \hbox{if }2<p<3,
		\\
		1 & \hbox{if }3\le p <6.
		\end{array}
		\right.
	\end{equation*}
Then \eqref{electrostatic} admits a nontrivial radial solution.
	\end{theorem}
We remark that stationary solutions coming from Theorem \ref{th:known} are definitely convincing for our theory, since both the energy
	\begin{multline}\label{eq:energy}
		{\mathcal E} (u(x)e^{-i\o t},\phi,{\bf 0}) \\= \frac 12 \irt \left[|\n u |^2 + |\n \phi|^2+(m^2+\o^2)u^2 -2e\o\phi u^2+e^2\phi^2u^2 - \frac 2 p |u|^{p}\right]\,dx
	\end{multline}
and the charge 
	\begin{equation}\label{eq:charge}
	{\mathcal Q} (u(x)e^{-i\o t},\phi,{\bf 0}) = e\irt (e\phi-\o)u^2\,dx
	\end{equation}
are finite (see \cite[Section 2.3]{F}). With an abuse of language, we will call {\it finite energy solution} a couple $(u,\phi)$ solving \eqref{electrostatic} and such that the energy related with \eqref{eq:matter} and \eqref{eq:elecmag} is finite. 

The strategy to achieve the result in Theorem \ref{th:known} consists in approaching \eqref{electrostatic} variationally and using usual tools of critical points theory to find solutions of the system as critical points of the functional
	\begin{equation*}
		I_\o(u,\phi) = \frac 12 \irt \left[|\n u |^2 - |\n \phi|^2+(m^2-\o^2)u^2 +2e\o\phi u^2-e^2\phi^2u^2\right]dx - \frac 1 p |u|^{p}\,dx
	\end{equation*} 
in $\HT\times\D$.\\

It is well known (see for example \cite{BF}) that the reduction method permits to convert the problem of finding critical points of $I_\o$ to the equivalent one of looking for critical points of the functional
	\begin{equation*}
J_\o(u)=\frac 1 2\irt [|\n u|^2+(m^2-\o^2)u^2+e\o\phi_u
u^2]\,dx-\frac 1 p \irt
|u|^p\, dx
\end{equation*}  
defined in $\HT$, where $\phi_u$ represents the unique function in $\D$ solving
	\begin{equation}\label{eq:phi_u}
		-\Delta \phi=e(\o-e\phi) u^2 
	\end{equation}
in the dual of $\D$. In order to promote the necessity of obtaining finite energy (and charge) solutions, the assumption $m>\o$ seems to arise quite naturally to implement our variational strategy. On the other hand, carrying out a deeper analysis of the model, we realize that such an assumption turns out to be purposely technical, since there is no physical motivation preventing the existence of finite energy standing waves having the form $\psi(x,t)=u(x)e^{-im t}$ (namely $m=\o$), under the gauge choice \eqref{FO}.\\
From the mathematical point of view, the idea of finding finite energy solutions to the system \eqref{electrostatic} in the limit case $m=\o$ appears immediately challenging.

Consider indeed the system
\begin{equation}    \label{main eq}\tag{$\mathcal P$}
\left\{
\begin{array}{ll}
-\Delta u +e(2\o-e\phi)\phi u-u^{p-1}=0 & \hbox{in }\RT,
\\
-\Delta \phi=e(\o-e\phi) u^2 & \hbox{in }\RT,\\
u >0, \phi >0& \hbox{in }\RT,
\end{array}
\right.
\end{equation}
for $e,$ $\o>0$ and $p> 1$.\\ 
The functional associated to the problem is 
	\begin{equation*}
				I_m(u,\phi) = \frac 12 \irt \left[|\n u |^2 - |\n \phi|^2 +2e\o\phi u^2-e^2\phi^2u^2\right] dx - \frac 1 p |u|^{p}\,dx
	\end{equation*}
whereas the formal reduced functional is
	\begin{equation*}
		J_m(u)=\frac 1 2\irt [|\n u|^2+e\o\phi_u u^2]\,dx-\frac 1 p \irt |u|^p\, dx.
	\end{equation*}
There is no difficulty in observing that, even if $I_m$ is of course well defined in $\HT\times\D$, the lack of an explicit expression of the $L^2$ norm of $u$ in $I_m$ makes it hard to apply standard critical points theory arguments in that space.\\
On the other hand, if we assume $\D\times\D$ as our setting, we have to face both the problem of inapplicability of the reduction method, and the difficulty in estimating the energy of any possible solution.\\
Finally, the introduction of an {\it ad hoc} functional setting as a sort of middle ground between those two spaces, does not immediately seem a feasible way.\\
A first attempt to solve \eqref{main eq} was made in \cite{APP}, by means of a perturbation argument (see also \cite{AP,BBS}). In that paper, the problem was considered in presence of an inhomogeneous nonlinearity in the first equation, and a solution $(u,\phi)$ in the sense of distributions was obtained in $\D\times\D$ as the limit of a sequence of solutions of approximating problems like these
\begin{equation}    \label{eq:pert}
\left\{
\begin{array}{ll}
-\Delta u +[\eps+e(2\o-e\phi)\phi] u-f(u)=0 & \hbox{in }\RT,
\\
-\Delta \phi=e(\o-e\phi) u^2 & \hbox{in }\RT.
\end{array}
\right.
\end{equation}
Unfortunately, the lack of information about the $L^2$ norm of $u$ did not permit to estimate the energy in order to confirm the validity of the model.

In this paper, we bridge this gap following an idea in \cite{IR} where the application of a comparison principle leads to show the exponential decay property possessed by our solution $u$.\\ 
Moreover, by similar arguments as those in \cite{R} and a new upper bound estimate holding uniformly for the  $L^p$ norm of suitable solutions of 
\begin{equation}    \label{eqeps}\tag{$\mathcal P_\eps$}
\left\{
\begin{array}{ll}
-\Delta u +[\eps+e(2\o-e\phi)\phi] u-u^{p-1}=0 & \hbox{in }\RT,
\\
-\Delta \phi=e(\o-e\phi) u^2 & \hbox{in }\RT,\\
u >0, \phi >0& \hbox{in }\RT,
\end{array}
\right.
\end{equation} as $\eps\to 0$, we are allowed to deal with a power-like nonlinearity.

The main result in this paper is the following.

	\begin{theorem}\label{main}
		For any $p\in \left(3, 6\right)$ there exists a finite energy solution $(u,\phi) \in C^2(\RT)\times C^2(\RT)$  to the problem \eqref{main eq}.
	\end{theorem}

The paper is organized in two sections.

In Section \ref{sec:pre} we provide an inequality useful to get a unifom upper bound on certain $L^2$ weighted norms. It will be used to control the $L^p$ norms of approximating solutions.

In Section \ref{sec:main} we prove our main Theorem \ref{main} by an argument based on the application of the comparison principle.

In what follows, the letter $C$ denotes a positive constant which may change from line to line. We also point out that, everytime we will handle radial functions, with an abuse of notation we will treat them as functions of one or three variables, denoting the argument by $r\in (0,+\infty)$ or $x\in\RT$. 
\section{Preliminary results}\label{sec:pre}
	As usual, for any $u\in \D\cap L^{\frac {12} 5}(\RT)$, we denote by $\phi_u$ the unique function in $\D$ satisfying \eqref{eq:phi_u} in a weak sense. \\
	It is well known that
		\begin{itemize}
			\item[1.] $\phi_u\ge 0$,
			\item[2.] $e\phi_u\le \o$ in the set $\{x\in\RT\mid u(x)\neq 0\}$,
			\item[3.] if $u$ is radial, then $\phi_u$ is radial.
		\end{itemize}
	We denote by $\Hr$ and $\Dr$ the spaces of radial functions respectively in $\HT$ and $\D$.  
	For all $M\ge 0$, define 
	$$\B_M:=\{u\in\Dr\cap L^{\frac {12} 5}(\RT)\mid \phi_u\in C^2(\RT)\hbox{ and } \|\n\phi_u\|_2\le M\}.$$
	The main object of this section is to prove the following result.

\begin{proposition}\label{pr:est}
		For all $\a>\frac 12$ and for all $M\ge 0$ there exists $C>0$ such that if $u\in \B_M$, then
			\begin{equation*}
				\irt \frac {u^2(x)}{\sqrt{|x|}(1+|\log |x||)^\a}\, dx\le C \left(\left(\irt |\n u|^2\, dx\right)^2+ \irt \phi_uu^2\,dx\right)^{\frac 12} .
			\end{equation*}
\end{proposition}

We start with some preliminary lemmas. The following one is essentially contained in \cite{R}; we recall it, adapted to our need.
	\begin{lemma}\label{le:uno}
		For all $\a>\frac 12$ there exists $C_\a>0$ such that for any $R>1$ and $h:\R_+\to\R_+$ measurable 
			\begin{equation*}
				\left(\int_{3R}^{+\infty}h(r)\, dr\right)^2\le C_\a \int_R^{+\infty}\left(\int_{\frac r2}^{2r}(1+\log r)^\a h(r)(1+\log s)^\a h(s)\,ds\right)\,dr.
			\end{equation*}
			\begin{proof}
				Let $\a >\frac 12$ and consider $R>1$. Set $k=[\log_2R]+1$, where $[\cdot]$ denotes the integer part. So $R<2^k<3R$. Following the idea in \cite{R}, we use the inequality $\left(\sum_{n=k}^\infty a_n\right)^2\le \sum_{n=k}^\infty\frac 1 {b_n} \sum_{n=k}^\infty b_na_n^2$ for $a_n=\int_{2^n}^{2^{n+1}}h(r)\,dr$ e $b_n=(1+n)^{2\a}$, and we have
					\begin{align*}
						\left(\int_{3R}^\infty h(r)\, dr\right)^2&\le \left(\sum_{n=k}^\infty\int_{2^n}^{2^{n+1}}h(r)\,dr\right)^2\\
						&\le \sum_{n=0}^\infty\frac 1 {(1+n)^{2\a}} \sum_{n=k}^\infty (1+n)^{2\a}\left(\int_{2^n}^{2^{n+1}}h(r)\,dr\right)^2\\
						&\le C_\a\sum_{n=k}^\infty\left[(1+n)^\a\int_{2^n}^{2^{n+1}}h(r)\,dr\right]^2.
					\end{align*}
				Finally, again as in \cite{R}, we have 
					\begin{multline*}
						\sum_{n=k}^\infty\left[(1+n)^\a\int_{2^n}^{2^{n+1}}h(r)\,dr\right]^2\\
						\le (\log_2e)^{2\a}\int_R^{+\infty}\left(\int_{\frac r 2}^{2r}(1+\log r)^\a h(r)(1+\log s)^\a
						h(s)\, ds\right)\, dr					
					\end{multline*}
				and we conclude.
			\end{proof}
	\end{lemma}
	\begin{lemma}\label{le:due}
			There exist two positive constants $C_1$ and $C_2$ such that for any $u\in \Dr\cap L^{\frac{12}5}(\RT)$ such that $\phi_u\in C^2(\RT)$ 
				\begin{equation*}
					\int_{\max(1,C_1\|\n\phi_u\|_2^2)}^{+\infty}\left(\int_{\max(1,C_1\|\n\phi_u\|_2^2)}^{+\infty}u^2(r)u^2(s)rs\min(r,s)\,ds\right)\,dr\le C_2 \irt\phi_uu^2\, dx.
				\end{equation*}
	\end{lemma}
	\begin{proof}
		Let $u$ be in $\Dr\cap L^{\frac {12}5}(\RT)$ and assume $\phi_u$ is in $C^2(\RT)$.\\
		By \cite[Radial Lemma A.III]{BL} we know that
		there exists $C>0$ which does not depend on $\phi_u$ such that for $|x|\ge 1$
		$$\phi_u(x)\le \frac C {\sqrt{|x|}}\|\n \phi_u\|_2.$$ Set $C_1=\frac{4e^2C^2}{\o^2}$.  
		Since $\phi_u$ satisfies
 		\begin{equation}\label{eq:ode}
 			-(r^2\phi_u')'=er^2(\o-e\phi_u) u^2
 		\end{equation}
 		in $[0,+\infty)$,
 		integrating \eqref{eq:ode} in $[0,t)$ for $t>\max (1,C_1 \|\n \phi_u\|_2^2)$, we get
 		\begin{align*}
 		\phi_u'(t)&=-\frac e{t^2}\int_0^t
 		s^2(\o-e\phi_u(s))u^2(s)\,ds\nonumber\\
 		&\le -\frac {e\o}{2t^2}\int_{\max(1,C_1\|\n\phi_u\|_2^2)}^t
 		s^2u^2(s)\,ds.
 		\end{align*}
 		Now, integrating in $(r,+\infty)$ for $r>\max(1,C_1\|\n\phi_u\|_2^2)$, we obtain
 			\begin{align}\label{eq:phi}
	 			\phi_u(r)&\ge \frac {e\o}{2}\int_r^{+\infty}\frac 1{t^2}\left(\int_{\max(1,C_1\|\n\phi_u\|_2^2)}^t
 			s^2u^2(s)\,ds\right)dt\nonumber\\
 							 &=\frac {e\o}{2}\int_{\max(1,C_1\|\n\phi_u\|_2^2)}^{+\infty}s^2u^2(s)\left(\int_{\max(r,s)}^{+\infty}\frac 1{t^2}\,dt\right)ds\nonumber\\
 							 &=\frac {e\o}{2}\int_{\max(1,C_1\|\n\phi_u\|_2^2)}^{+\infty}\frac{s^2}{\max(r,s)}u^2(s)ds.
 			\end{align}
 		Finally, multiplying by $r^2u^2(r)$ and integrating in $\big(\max(1,C_1\|\n\phi_u\|_2^2),+\infty\big)$, we have
 			\begin{align*}
	 			\int_0^{+\infty}r^2\phi_u(r)u^2(r)\,dr &\ge\frac {e\o}{2} \int_{\max(1,C_1\|\n\phi_u\|_2^2)}^{+\infty}\left(\int_{\max(1,C_1\|\n\phi_u\|_2^2)}^{+\infty}\frac{r^2s^2}{\max(r,s)}u^2(r)u^2(s)\,ds\right)\!dr\\
	 			&=\frac {e\o}{2}\int_{\max(1,C_1\|\n\phi_u\|_2^2)}^{+\infty}\left(\int_{\max(1,C_1\|\n\phi_u\|_2^2)}^{+\infty}\!\!\!rs\min(r,s)u^2(r)u^2(s)\,ds\right)\!dr
 			\end{align*}
 		from which we conclude.
	\end{proof}
	\begin{lemma}\label{le:weight}
			For any $\a>\frac 12$ and $R_0>1$ there exists $C>0$ such that for any measurable function $u:\R_+\to\R$ 
				\begin{multline*}
	\int_0^{+\infty}\!\!\frac{u^2(r)r^{\frac 32}}{(1+|\log r|)^\a}\, dr\\
	\le C \left[\left(\irt|\n u|^2\, dx\right)^2+\int_{R_0}^{+\infty}\left(\int_{R_0}^{+\infty}u^2(r)ru^2(s)s\min(r,s)\,ds\right)dr\right]^{\frac 12}.
				\end{multline*}
	\end{lemma}
	\begin{proof}
		Let $\a >\frac12$ and $R_0 > 1$ and consider $C_\a>0$ as in Lemma \ref{le:uno} 
		
		Applying Lemma \ref{le:uno} for $h(r)=\frac{u^2(r)r^{\frac 32}}{(1+|\log r|)^\a}$, by Holder and Gagliardo-Nirenberg inequalities, we have that
		\begin{align*}
		\left(\int_0^{+\infty}\!\!\frac{u^2(r)r^{\frac 32}}{(1+|\log r|)^\a}\, dr\right)^2\!\!\!\!&\le 2\left(\int_0^{6R_0}\frac{u^2(r)r^{\frac 32}}{(1+|\log r|)^\a}\, dr\right)^2+2 \left(\int_{6R_0}^{+\infty}\frac{u^2(r)r^{\frac 32}}{(1+|\log r|)^\a}\, dr\right)^2\\
		&\le 2 \left(\int_{0}^{6R_0}u^6(r)r^2\,dr\right)^{\frac 23}\left(\int_{0}^{6R_0}r^{\frac 54}\,dr\right)^{\frac 43}\\
		&\qquad + 2C_\a \int_{2R_0}^{+\infty}\left(\int_{\frac r2}^{2r}u^2(r)r^{\frac 32}u^2(s)s^{\frac 32}\,ds\right)dr\\
		&\le C \left(\irt|\n u|^2\, dx\right)^2\\
		&\qquad+2\sqrt 2 C_\a\int_{2R_0}^{+\infty}\left(\int_{\frac r2}^{2r}u^2(r)ru^2(s)s\min(r,s)\,ds\right)dr\\
		&\le  C \left(\irt|\n u|^2\, dx\right)^2\\
		&\qquad+2\sqrt 2 C_\a\int_{R_0}^{+\infty}\left(\int_{R_0}^{+\infty}u^2(r)ru^2(s)s\min(r,s)\,ds\right)dr.
		\end{align*}
		In the third inequality we have used that $\sqrt{rs}\le \sqrt 2\min (r,s)$ in the set $\{(r,s)\in\R_+\times\R_+\mid \frac r 2 \le s \le 2r\}$.
	\end{proof}
Now we are ready to prove Proposition \ref{pr:est}
\begin{proof}[Proof of Proposition \ref{pr:est}]
	Consider $\a >\frac 12$ and $M\ge 0$. Take $C_1$ and $C_2$ positive constants as in Lemma \ref{le:due} and $R_0>\max(1, C_1 M^2)$. The conclusion follows from Lemma \ref{le:due} and Lemma \ref{le:weight}.
\end{proof}
By Lemma \ref{le:weight} we also deduce the following estimate on the Lebesgue norms.
	\begin{proposition}\label{pr:lp}
			Let $q\in (\frac{18}7, 6]$ and $R_0>1$.  Then there exists $C>0$ such that for every radial and measurable $u:\RT\to \R$  we have
			\begin{equation*}
				\left(\irt |u|^q\, dx\right)^{\frac 1 q}\le C \left[\irt |\n u|^2\, dx+\left(\int_{R_0}^{+\infty}\left(\int_{R_0}^{+\infty}u^2(r)ru^2(s)s\min(r,s)\,ds\right)dr\right)^{\frac 12}\right]^{\frac 12}
			\end{equation*}
	\end{proposition}
	\begin{proof}
		Fix $q\in(\frac{18}7, 6]$ and $R_0>1$. By continuous embedding theorems proved in \cite{SWW1,SWW2}, there exists $\eta >\frac 12$ such that the following inequality holds for some positive constant $C$ and any $u\in \Dr\cap L^2(\RT, V(x)dx)$, where $V(x)=\frac 1{1+|x|^\eta}$:
			\begin{equation*}
				\left(\irt |u|^q\, dx\right)^{\frac 1 q}\le C \left(\irt |\n u|^2\,dx + \irt \frac{u^2}{1+|x|^\eta}\, dx\right)^{\frac 12}.
			\end{equation*}
		Now choose $\a >\frac 12 $. Since 
		$$\lim_{|x|\to 0}\frac{\sqrt{|x|}(1+|\log|x||)^\a}{1+|x|^\eta}=0=\lim_{|x|\to +\infty}\frac{\sqrt{|x|}(1+|\log|x||)^\a}{1+|x|^\eta},$$ 
		we deduce that
			\begin{equation*}
				\left(\irt |u|^q\, dx\right)^{\frac 1 q}\le C \left(\irt |\n u|^2\,dx + \irt \frac{u^2}{\sqrt{|x|}(1+|\log|x||)^\a}\, dx\right)^{\frac 12}													
			\end{equation*} 
		and then the conclusion follows by Lemma \ref {le:weight}.
	\end{proof}

\section{Existence of a finite energy solution}\label{sec:main}
In this section we assume $p\in(3, 6)$.\\
Looking at the proof of Theorem \ref{th:known} (see  in \cite{APP,BF,DM}), we have that for any $\eps>0$ there exists a solution $(u_\eps,\phi_{u_\eps})$ to the problem \eqref{eqeps} (positiveness can be deduced by standard arguments based on the maximum principle), where $u_\eps$ is found as a critical point of the functional 
\begin{equation*}
J_\eps(u):=\frac 1 2\irt [|\n u|^2+\eps u^2+e(2\o-e\phi_u(x) )\phi_u
u^2]\,dx-\frac 1 p \irt
|u|^p\, dx
\end{equation*}
and, called $c_\eps$ the mountain pass level of $J_\eps$, we have $J_\eps(u_\eps)\le c_\eps$.

	Let $\eps_n\to 0$ and consider the sequence $(u_n,\phi_n)\in \Hr\times\Dr$ of solutions of \eqref{eqeps} built by Theorem \ref{th:known} for $\eps=\eps_n$. We set $J_n=J_{\eps_n}$ and call $c_n$ the corresponding mountain pass level.\\ 
	By standard elliptic arguments, we can prove that both $u_n$ and $\phi_n$ are in $C^2(\RT)$, for any $n\ge 1$. \\
	Moreover we have the following result about boundedness of the sequence.
		\begin{proposition}\label{pr:bound}
				The sequence $(u_n)_n$ is bounded in $\Dr$ and  there exists $C>0$ such that $\irt e\o\phi_nu_n^2\, dx\le C$ for every $n\ge 1.$ Moreover, there exists $M\ge 0$ such that $(u_n)_n$ is a sequence in $\mathcal B_M$.
		\end{proposition}
		\begin{proof}
				If $p\in (3,4]$, then, proceeding as in \cite[Theorem 1.1]{APP}, by suitably combining the inequality $I_n(u_n)\le c_n$ with Nehari and Pohozaev identities, we have that for any $\g\in\R$
				\begin{equation*}
				D_{p,\g}\irt |\n u_n|^2\, dx+\irt [C_{p,\g}\eps_n + B_{p,\g}e\o\phi_n+A_{p,\g}e^2\phi_n^2]u_n^2\le c_n
				\end{equation*}
			where 
				\begin{align*}
					A_{p,\g}&=\frac{1+2\g(p-3)}{p},\\
					B_{p,\g}&=\frac{p-10\g p-4+24\g}{2p},\\
					C_{p,\g}&=\frac{(p-2)(1-6\g)}{2p},\\
					D_{p,\g}&=\frac{p-2p\g -2+12\g}{2p}.
				\end{align*}
			From the same computations as those in \cite[Lemma A.1]{APP}, we deduce that for 
			$\g\in \left(\frac{2-p}{2(6-p)},\frac{4-p}{24-10p}\right)$ we have 			\begin{align*}
				A_{p,\g}&>0,\\
				B_{p,\g}&>0,\\
				C_{p,\g}&>0,\\
				D_{p,\g}&>0,
			\end{align*}
			and then, since it is a simple exercise to see that the sequence $(c_n)_n$ is bounded above, we conclude that both $(\|\n u_n\|_2)_n$ and $\left(\irt \phi_nu_n^2\,dx\right)_n$ are bounded.\\
			If $p\in (4,6)$, then we proceed as in \cite{DM} to obtain again boundedness of $(\|\n u_n\|_2)_n$ and $\left(\irt \phi_nu_n^2\,dx\right)_n$.\\
			Since by the second equation in \eqref{eq:pert} we have 
				\begin{equation*}\label{eq:boun}
					\irt |\n\phi_n|^2\, dx +e^2\irt \phi_n^2u_n^2\, dx =\irt e\o\phi_nu_n^2\, dx,
				\end{equation*}
			we can establish the existence of some $M\ge 0$ such that every element in the sequence $(u_n)_n$ is in $\mathcal B_M$. 
		\end{proof}
	Now we proceed using the same notations and arguments as in \cite{R}: by Proposition \ref{pr:est}, we deduce that for any $\eta >\frac 12$ the sequence $(u_n)_n$ is bounded in $\Dr\cap L^2(\RT, V(x)dx)$, where $V(x)=\frac 1{1+|x|^\eta}$ and, by embedding theorems proved in \cite{SWW1,SWW2}, $(u_n)_n$ is bounded also in $L^q(\RT)$ for all $q\in (\frac {18}7, 6]$. By standard compactness arguments based on the decay estimate holding for functions in $\Dr$ (see \cite[Radial Lemma A.III]{BL}), this embedding is compact for $q\in(\frac{18}7,6)$. Then, up to a subsequence, we have that there exists $u_0\in\Dr$ such that
	\begin{equation}\label{eq:weak}
	u_n\rightharpoonup u_0\hbox{ in }\Dr
	\end{equation} 
	and, for every $q\in (\frac {18}7, 6)$, $u_0\in L^q(\RT)$ and 
	\begin{equation*}
		u_n\to u_0 \hbox{ in } L^q({\RT}).
	\end{equation*}
	In particular, 
	\begin{equation}\label{eq:str}
		u_n\to u_0 \hbox{ in } L^p({\RT}).
	\end{equation}
	By Proposition \ref{pr:bound}, the sequence $(u_n)_n$ is in $\mathcal B_M$ for some $M\ge 0$, and then there exists $\phi_0\in\Dr$ such that, up to a subsequence, 
		\begin{equation}\label{eq:phin}
			\phi_n\rightharpoonup \phi_0\hbox{ in }\Dr.
		\end{equation}
	Of course, $u_0\ge 0$ and $\phi_0\ge 0$. Moreover $u_0\neq 0$ by \eqref{eq:str} and the following result
		\begin{proposition}
			There exists $C>0$ such that $\|u_n\|_p\ge C$ uniformly.
		\end{proposition}
			\begin{proof}
				Here we will follow an idea in \cite{IR}. Set $R_0> \max (1, C_1M^2)$, where $C_1>0$ is the same as in Lemma \ref {le:due}, and define 
					\begin{equation*}
						M[u]=\irt |\n u|^2\, dx + \int_{R_0}^{+\infty}\left(\int_{R_0}^{+\infty}u^2(r)ru^2(s)s\min(r,s)\, ds\right)dr
					\end{equation*}
				and
					\begin{equation*}
						N[u]=\left[\irt |\n u|^2\, dx+\left(\int_{R_0}^{+\infty}\left(\int_{R_0}^{+\infty}u^2(r)ru^2(s)s\min(r,s)\,ds\right)dr\right)^{\frac 12}\right]^{\frac 12}.
					\end{equation*}
				Take $C>0$ as in Proposition \ref{pr:lp} in correspondence of $p$ and  $R_0$ and call $u_n^t(x)=t^2u_n(tx)$ for all $t>0$. We have that
					\begin{equation}\label{eq:comp}
						\irt |u_n|^p\,dx= t^{3-2p}\irt |u_n^t|^p\,dx\le C t^{3-2p}N[u^t_n].
					\end{equation} 
				Now, for any $n\ge 1$, set $t_n=(M[u_n])^{-\frac 13}$ so that $M[u_n^{t_n}]=t_n^3 M[u_n]=1$.
				Since for every $v:\RT\to \R$ measurable we have that $M[v]\le1$ implies $\frac 12 (N[v])^4\le M[v]$, 
				from \eqref{eq:comp} we deduce
					\begin{equation}\label{eq:comp2}
						\irt |u_n|^p\,dx\le C t_n^{3-2p}N[u^{t_n}_n]\le C \sqrt[4]2(M[u_n])^{\frac{2p-3}3}.
					\end{equation} 
				Now, by our choice of $R_0$ and the fact that $u_n\in\mathcal B_M$ for every $n\ge 1$, by Lemma \ref{le:due} and \eqref{eq:comp2} we infer that for every $n\ge 1$
					\begin{equation*}
						\irt |u_n|^p\, dx \le C \left(\irt |\n u_n|^2\, dx + \irt \phi_nu_n^2\, dx\right)^{\frac{2p-3}{3}}.
					\end{equation*}
				Since the functions $u_n$ satisfy the Nehari identity, we have
					\begin{align*}
					\irt [|\n u_n|^2+e\o\phi_n
					u_n^2]\,dx &\le \irt [|\n u|^2+\eps_n u_n^2+e(2\o-e\phi_n(x) )\phi_n
					u_n^2]\,dx\\
					&= \irt
					|u_n|^p\, dx \le C \left(\irt |\n u_n|^2\, dx + \irt \phi_nu_n^2\, dx\right)^{\frac{2p-3}{3}}.
					\end{align*}
				The conclusion follows recalling the fact that $p>3$.
			\end{proof}
	
		\begin{proposition}\label{pr:dist}
		The couple $(u_0,\phi_0)$ solves \eqref{main eq} in the sense of distributions, namely
			\begin{align*}
				&\irt(\n u_0\n\psi +e(2\o-e\phi_0)\phi_0u_0\psi-u_0^{p-1}\psi)\,dx=0,
				\\
				&\irt\n \phi_0\n\psi\, dx=\irt e(\o-e\phi_0) u_0^2\psi\,dx,
			\end{align*}
		for all $\psi\in C^{\infty}_0(\RT)$.
		\end{proposition}
		Taking into account \eqref{eq:weak}, \eqref{eq:str} and \eqref{eq:phin}, the proof is definitely similar to that of \cite[Theorem 1.2]{APP}, so we omit it. Moreover, since for all $q\in\left(\frac{18}7, 6\right)$ we have $u_0\in\D\cap L^{q}(\RT)$,  a direct application of Holder inequality and standard density arguments show the following
		\begin{corollary}\label{cor:weak}
			For all $v\in\D\cap L^h(\RT)$ and $w\in\D\cap L^k(\RT)$ with $(h,k)\in\left[1,\frac{9}{4}\right)\times \left[1,\frac 92\right)$
				\begin{align*}
				&\irt(\n u_0\n v +e(2\o-e\phi_0)\phi_0u_0 v-u_0^{p-1} v)\,dx=0,
				\\
				&\irt\n \phi_0\n w\, dx=\irt e(\o-e\phi_0) u_0^2 w\,dx.
				\end{align*}
		\end{corollary}

	We emphasize the fact that, until this step, we can not establish a stronger form of relationship between the couple $(u_0,\phi_0)$ and the equations in \eqref{main eq}. 
	
	In particular, the possibility that our particle possesses infinite energy should compromise our theory making it inconsistent with any physical purpose. To this end, first we prove the following generalization of the Strauss' radial Lemma \cite{strauss}
	
	\begin{lemma}\label{le:Str}
		Let $N\ge 3$ and  $q\in [2,2N).$ Then there exists $C>0$ such that for any $u\in\mathcal D^{1,2}(\RN)\cap L^{q}(\RN)$ and $|x|>1$ we have
		\begin{equation*}
		|u(x)| \le C\frac{\|\n u\|_2 + \|u\|_q}{|x|^{\frac{2N-q}{2q}}}.
		\end{equation*}
		Moreover $u$ is almost everywhere equal to a continuos function in $\RN\setminus\{0\}$. 
	\end{lemma}
	\begin{proof}
		Set $u\in C_0^{\infty}(\RT)$ and consider $N\ge 3$. As in \cite{strauss}, we	obtain the inequality
		\begin{equation*}
		r^{N-1}u^2(r)\le \int_0^r [(u'(s))^2+u^2(s)]s^{N-1}\,ds + m r^{N-2}u^2(r)
		\end{equation*}
		where $m=\frac{N-1}2$.
		Now we proceed with the following estimate 
		\begin{align*}
		\int_0^r u^2(s) s^{N-1}\, ds &= \int_0^r u^2(s) s^{\frac{2(N-1)}{q}}s^{\frac{(N-1)(q-2)}q}\,ds\\
		&\le \left(\int_0^r |u(s)|^q s^{N-1}\, ds \right)^{\frac 2 q}\left(\int_0^r s^{N-1}\, ds\right)^{\frac{q-2}q},
		\end{align*}
		and then, comparing the two inequalities, we arrive to
		\begin{equation*}
		r^{N-1}u^2(r)\le C\left(\|\n u\|_2^2+ r^{\frac{N(q-2)}q}\|u\|_q^2\right)+mr^{N-2}u^2(r),
		\end{equation*}
		that is 
		\begin{equation*}
		\left(1-\frac m r \right)u^2(r)\le \frac C{r^\frac{2N-q}q} \big(\|\n u\|_2+\|u\|_q\big)^2,
		\end{equation*}
		corresponding to our estimate. We conclude by density arguments.
	\end{proof}

	Now we can prove the following integrability result.
	
	\begin{proposition}\label{pr:crucial}
		The function $u_0$ is in $L^2(\RT)$.
	\end{proposition}
	\begin{proof}
		In this proof we combine ideas in \cite{IR} and \cite{AP} adapting them to our not trivial situation.
		
		By contradiction, assume that $\|u_0\|_2=+\infty$. By Proposition \ref{pr:bound} and \eqref{eq:phi}, we know that there exists $R_1>0$ such that for any $n\ge 1$ and $r>R_1$
		\begin{equation}\label{eq:lower}
		\phi_n(r)\ge \frac {e\o}{2}\int_{R_1}^{+\infty}\frac{s^2}{\max(r,s)}u_n^2(s)ds.
		\end{equation}
		Now, for every $r, s$ with $0<r<s$, we set $ A_r^s= B_s\setminus B_r$, where $B_r$ and $B_s$ are the balls centered in $0$ and with radius respectively $r$ and $s$.
		Since 
		\begin{equation*}
		u_n\to u_0 \hbox{ in }L^2(A_r^s)
		\end{equation*}
		for every $r < s$, we have that for any $K>0$ there exists $R_K>0$ for which 
		\begin{equation*}\label{eq:lim}
		\lim_n \|u_n\|_{L^2(A_{R_1}^{R_K})}>K.
		\end{equation*}
		Then by \eqref{eq:lower} we have that there exist three positive numbers $C$, $R_1$ and $R_2$  and $n_0\in \N$ such that $R_1<R_2$ and
		\begin{equation*}\label{eq:phinr}
		\phi_n(r)\ge \frac {e\o}{2}\int_{R_1}^{+\infty}\frac{s^2}{\max(r,s)}u_n^2(s)ds\ge \frac {e\o}{2r}\int_{R_1}^{R_2}s^2u_n^2(s)ds> \frac{C} r
		\end{equation*}
		for every $r>R_2$ and $n\ge n_0$. Since, up to a subsequence, $\phi_n\to\phi_0$ pointwise, we deduce that
		\begin{equation}\label{eq:phir}
			\phi_0(r)\ge \frac{C}{r},\hbox{ for }r>R_2
		\end{equation} 
		and, of course, $\o-e\phi_0\ge 0$.
		Since, exactly as in \cite[Theorem 6.1]{IR}, we have that for every $R>0$ there exists $\bar R>R$ such that $u_0(\bar R)\le \phi_0(\bar R)$, we can consider $\bar R > R_2$ such that the function $\varphi$ defined as follows
			\begin{equation*}    \label{varphi}
			\varphi(x)=\left\{
			\begin{array}{ll}
			0 & \hbox{if }|x|<\bar R,
			\\
			(u_0-\phi_0)_+& \hbox{if }|x|\ge \bar R
			\end{array}
			\right.
			\end{equation*}
		is in $\Dr \cap L^{q}(\RT)$ for any $q\in \left(\frac{18}{7}, 6\right]$, and by Lemma \ref{le:Str} and \cite[Radial Lemma A.III]{BL} we have 
			\begin{equation}\label{eq:neg}
				e(\o-e\phi_0) u_0^2 - u_0^{p-1}>0 \hbox{ in }[\bar R, +\infty[.
			\end{equation}
		Since $\frac 94<\frac{18}7<\frac 92$, by Corollary \ref{cor:weak} the function $\varphi$ is a test function for the second equation, but it is not for the first. Then we approximate it by means of a family of cut off functions in the following way.\\
		Define $k:\RT\to[0,1]$ as a smooth radial function, radially decreasing and such that $k\equiv 1$ in $|x|\le 1$ and $k\equiv 0$ in $|x|\ge 2.$
		For any $M>0$, define $v_M= k_M \varphi $, where $k_M(x)=k(x/M)$. Of course $v_M\ge 0$ and since $\supp v_M$ is compact and $\n \varphi_M\in L^2(\RT)$, we have that $\varphi_M$ is a test function for both the equations in the system. Moreover
			\begin{equation*}
				\varphi_M\to \varphi\hbox{ in }L^q \hbox{ for all } q\in \left(\frac{18}{7}, 6\right]
			\end{equation*}
		and, taken an arbitrary $h\in \left(\frac{18}{7}, 6\right)$,
			\begin{align*}
				\|\n \varphi-\n \varphi_M\|^2_2&\le C \int_{|x|\ge M}|\n \varphi|^2\,dx+ \frac C{M^2}\int_{A_M^{2M}} \varphi^2\, dx\\
				&\le o_M(1)+ \frac C{M^2}\|\varphi\|_h^2|A_M^{2M}|^{\frac{h-2}{h}}\\
				&\le o_M(1)+ \frac C{M^{\frac{6-h}h}}\|\varphi\|_h^2,
			\end{align*}
		and so
			\begin{equation*}
				\varphi_M\to \varphi \hbox{ in } \Dr.
			\end{equation*}
		By Corollary \ref{cor:weak},
		\begin{align*}
			&\irt (\n u_0\n \varphi_M +  e(2\o-e\phi_0)\phi_0 u_0\varphi_M- u_0^{p-1}\varphi_M)\, dx=0 ,
			\\
			&\irt\n \phi_0\n \varphi_M\, dx=\irt e(\o-e\phi_0) u_0^2\varphi_M\, dx
		\end{align*}
		so, comparing and using the fact that $(2\o-e\phi_0)\phi_0 u_0\varphi_M\ge 0$,
		\begin{align*}
			\irt \n (u_0-\phi_0)\n \varphi_M \, dx&=\irt [-   e(2\o-e\phi_0)\phi_0 u_0 -e(\o-e\phi_0) u_0^2 + u_0^{p-1}]\varphi_M\, dx\\
														  &\le\irt  [-e(\o-e\phi_0) u_0^2 + u_0^{p-1}]\varphi_M\, dx.
		\end{align*}
		Letting $M$ go to $+\infty$, by continuity we have
			\begin{align*}
				\irt \n(u_0-\phi_0)\n \varphi\, dx\le \irt  [-e(\o-e\phi_0) u_0^2 + u_0^{p-1}]\varphi\, dx.
			\end{align*}
		By definition of $\varphi$ and \eqref{eq:neg}, we deduce
			\begin{align*}
				\int_{|x|\ge \bar R}|\n (u_0-\phi_0)_+|^2\, dx&=\irt \n(u_0-\phi_0)\n \varphi\, dx\\
				&\le \irt  [-e(\o-e\phi_0) u_0^2 + u_0^{p-1}]\varphi\, dx\\
				&= \int_{|x|\ge \bar R}  [-e(\o-e\phi_0) u_0^2 + u_0^{p-1}](u_0-\phi_0)_+\, dx\le 0
			\end{align*}
		and then $u_0\le\phi_0$ in $(\bar R,+\infty)$.\\
		Now, possibly replacing $\bar R$ with a larger value, by Lemma \ref{le:Str} we can assume $u_0^{p-3}(r)<\frac{e\o}2$ in $(\bar R, +\infty)$, so that, by \eqref{eq:phir},
			\begin{equation*}
				 e(2\o-e\phi_0(r))\phi_0(r) - u_0^{p-2}(r)\ge e\o\phi_0(r) - \frac{e\o}{2}u_0(r)\ge \frac{e\o}2\phi_0(r)\ge \frac{e\o C}2\frac 1r
			\end{equation*}
		Take $\g\in \left(0,\frac{Ce\o}2\right)$ consider the problem
			\begin{equation*}
		\begin{cases}
		\dis -\Delta w 
		+ \frac{\g}{|x|} 
		w =
		0 &\hbox{ if } |x|>\bar R, 
		\\[3mm]
		w=u_0&\hbox{ if } |x|=\bar R,\\[3mm]
		w\to 0 &\hbox{ as }|x| \to +\infty
		\end{cases}
		\end{equation*}
		and let $v$ be a radial solution.
		Now we again use the comparison principle by approximation. Consider the function $\psi:\RT\to\R$ such that
					\begin{equation*}    \label{psi}
						\psi(x)=\left\{
						\begin{array}{ll}
						0 & \hbox{if }|x|<\bar R,
						\\
						(u_0-v)_+& \hbox{if }|x|\ge \bar R
						\end{array}
						\right.
					\end{equation*} 
		As before, define $\psi_M=k_M\psi$ and multiply the first equation of the system and equation
			$$-\Delta v+\frac{\g}{|x|}v=0$$ by $\psi_M$ (which is a test function for both the equations) and integrate. Comparing, we obtain
			\begin{multline*}
				\irt \n (u_0-v)\n \psi_M\, dx + \irt \frac \g{|x|}(u_0-v)\psi_M\, dx\\
				=\irt  \left(u_0^{p-2}-e(2\o - e\phi_0)\phi_0+ \frac{\g}{|x|} \right)u_0\psi_M\, dx.
			\end{multline*}
		Observe that for any $M\ge \bar R$ it is 
			$$\irt \frac \g{|x|}(u_0-v)\psi_M\, dx= \int_{A_{\bar R}^{2M}}k_M\frac\g{|x|}|(u_0-v)_+|^2\, dx\ge 0$$
		and 
			\begin{multline*}
				\irt \left(u_0^{p-2}-e(2\o - e\phi_0)\phi_0+ \frac{\g}{|x|} \right)u_0\psi_M\, dx\\
				=\int_{A_{\bar R}^{2M}}k_M \left(u_0^{p-2}-e(2\o - e\phi_0)\phi_0+ \frac{\g}{|x|} \right)u_0(u_0-v)_+\, dx\le 0.
			\end{multline*}
		Then 
		\begin{equation*}
			\irt \n (u_0-v)\n \psi_M\, dx\le 0
		\end{equation*}
		and, passing to the limit as $M$ goes to infinity, we have
			\begin{equation*}
				\int_{|x|\ge \bar R} |\n (u_0-v)_+|^2\, dx \le 0,
			\end{equation*}
		and then $u_0\le v$ almost everywhere in $(\bar R,+\infty)$. The contradiction arises since $v$  exponentially decays at infinity (see \cite{AMR})).

	\end{proof}

Finally we conclude with the following

	\begin{proof}[Proof of Theorem \ref{main}]
		Since $u_0\in L^2(\RT)$, by Proposition \ref{pr:dist} and a density argument, 
		we have that for all $(v,w)\in\HT\times\D$
			\begin{align*}
				&\irt(\n u_0\n v +e(2\o-e\phi_0)\phi_0u_0 v-u_0^{p-1} v)\,dx=0,
				\\
				&\irt\n \phi_0\n w\, dx=\irt e(\o-e\phi_0) u_0^2 w\,dx.
			\end{align*}
			We deduce that
				\begin{itemize}
					\item[1.] by uniqueness  $\phi_0=\phi_{u_0}$,
					\item[2.] by ellipticity $u_0\in C^2(\RT)$ and $\phi_{u_0}\in C^2(\RT)$, and equations are satisfied pointwise,
					\item[3.] by the strong maximum principle $u_0>0$ and $\phi_{u_0}>0$,
					\item[4.] by Berestycki - Lions' radial lemma $\phi_{u_0}$ satisfies \eqref{FO},
					\item[5.] since $(u_0,\phi_{u_0})\in\Hr\times\Dr$, the energy \eqref{eq:energy} and the charge \eqref{eq:charge} are finite.
				\end{itemize}
	\end{proof}
	\begin{remark}\label{re:cou}
		Observe that, since $u_0\in L^2(\RT)$ and $\phi_{u_0}\in C^2(\RT)$, the function $\phi_{u_0}$ satisfies \eqref{eq:ode} for $u=u_0$ and then, by direct computations, we deduce that there exist two positive constants $K_1$ and $K_2$ such that for any $r\ge 1$,
			\begin{equation*}
				\frac{K_1}{r}\le \phi_{u_0}(r)\le \frac{K_2}{r}.
			\end{equation*}
			
		By this fact and using the same arguments as those in the proof of Proposition  \ref{pr:crucial} (actually we do not need anymore truncations to apply the comparison principle), we show that $u_0$ decays exponentially at infinity.
		
		We conclude that the majority of standing wave's charge is localized inside a bounded region, in line with the particle-like interpretation. 
	\end{remark}


\begin{thebibliography}{99}
	
	\bibitem{AMR}
	A. Ambrosetti, A. Malchiodi, D. Ruiz, {\it Bound states of nonlinear Schr\"odinger
equations with potential vanishing at infinity}, J. Anal. Math., {\bf 98}, (2006), 317--34.
	
	\bibitem{APP}
	A. Azzollini, L. Pisani, A. Pomponio, {\it Improved estimates and a limit case for
		the electrostatic Klein-Gordon-Maxwell system}, Proc. Roy. Soc. Edinburgh Sect. A, {\bf 141}, (2011), 449--463.

	\bibitem{AP}
	A. Azzollini, A. Pomponio, {\it Positive energy static solutions for the Chern-Simons-Schr\"odinger system under a large-distance fall-off requirement on the gauge potentials}, preprint.

	\bibitem{BBS}
	J. Bellazzini, C. Bonanno, G. Siciliano, {\it Magneto-static vortices in two dimensional abelian gauge theories},  Mediterr. J. Math., {\bf 6}, (2009), 347--366.
	
	\bibitem {BF}
	V. Benci, D. Fortunato, {\it Solitary waves of the nonlinear
		Klein-Gordon field equation coupled with the Maxwell equations},
	Rev. Math. Phys., {\bf 14}, (2002), 409--420.
	
	\bibitem{BF2}
	V. Benci, D. Fortunato, {\it Solitary waves in classical field theory}, in Nonlinear Analysis and Applications to Physical Sciences, V. Benci, A. Masiello Eds Springer, Milano (2004), 1-50. 
	
	\bibitem{BL}
	H. Berestycki, P.L. Lions, {\it Nonlinear scalar field equations. I. Existence of
	a ground state}, Arch. Rational Mech. Anal., {\bf 82}, (1983), 313--345.
	
	\bibitem {DM}
	T. D'Aprile, D. Mugnai, {\it Solitary waves for nonlinear
		Klein-Gordon-Maxwell and Schr\"{o}dinger-Maxwell equations},
	Proc. Roy. Soc. Edinburgh Sect. A, {\bf 134}, (2004), 893--906.
	
	\bibitem{DM2}
	T. D'Aprile, D. Mugnai, {\it Non-existence results for the coupled Klein-Gordon-Maxwell equations},
	Adv. Nonlinear Stud., {\bf 4}, (2004), 307--322.
	
	\bibitem{dps}
	P. d'Avenia, L. Pisani, G. Siciliano, {\it Dirichlet and Neumann problems for Klein-Gordon-Maxwell systems},
	Nonlinear Anal., {\bf 71}, (2009), 1985--1995.
	
	\bibitem{dps1}
	P. d'Avenia, L. Pisani, G. Siciliano, {\it Klein-Gordon-Maxwell
		systems in a bounded domain}, Discrete Contin. Dyn. Syst., {\bf 26},
	(2010), 135--149.
	
	\bibitem{F}
	D. Fortunato, {\it Solitary waves and electromagnetic field}, Boll. UMI {\bf 9}, (2008), 767--789.
	
	\bibitem{IR} 
	I. Ianni, D. Ruiz, {\it Ground and bound states for a static Schr\"odinger-Poisson-Slater problem}, 
	Commun. Contemp. Math., {\bf 14}, (2012), 22 pp.
	
	\bibitem{R}
	D. Ruiz,
	{\it On the Schr\"odinger-Poisson-Slater system: Behavior of minimizers, radial and nonradial cases}, Arch. Ration. Mech. Anal. {\bf 198}, (2010), 349--368.
	
	\bibitem{strauss}
	W.A.~Strauss,
	{\em Existence of solitary waves in higher dimensions}, 
	Comm. Math. Phys. {\bf 55} (1977), 149--162.
	
	\bibitem{SWW1}
	J. Su, Z.-Q. Wang, M. Willem, {\it Nonlinear Schr\"odinger equations with unbounded and decaying radial potentials}, Commun. Contemp. Math. {\bf 9}, (2007), 571--583.
	
	\bibitem{SWW2}
	J. Su, Z.-Q. Wang, M. Willem, {\it Weighted Sobolev embedding with unbounded and decaying radial potentials}, J. Differential Equations, {\bf 238}, (2007), 201--219.
	

	
\end{thebibliography}
\end{document}